\documentclass[11pt,reqno]{amsart}

\usepackage{amssymb,amsmath}
\usepackage{amsfonts}
\usepackage{graphicx}
\usepackage{color}

\newcommand{\R}{{\mathbb R}}
\newcommand{\D}{\mathcal{D}}
\def\sign{\operatorname{sign}}
\newcommand{\ci}[1]{_{{}_{\scriptstyle{#1}}}}

\newcommand{\Bel}{\mathbf B}
\newcommand{\av}[2]{\langle #1\rangle\ci {#2}}
\newcommand{\ve}{\varepsilon}
\newcommand{\vf}{\varphi}

\newtheorem{theorem}{Theorem}
\newtheorem{lemma}[theorem]{Lemma}
\newtheorem*{cor}{Corollary}
\numberwithin{equation}{section}

\newcommand{\BMO}{{\rm{BMO}}}
\newcommand{\eq}[2][label]{\begin{equation}\label{#1}#2\end{equation}}
\newcommand{\df}{\stackrel{\mathrm{def}}{=}}

\begin{document}

\title[Sharp constants in the the John--Nirenberg inequality]
{{Sharp constants in the classical weak form of the John--Nirenberg inequality}}
\author{Vasily Vasyunin}
\address{Vasily Vasyunin,
\newline{St.-Petersburg Department of V. A. Steklov Mathematical
Institute}
\newline{\tt vasyunin@pdmi.ras.ru}}
\author{Alexander Volberg}
\address{Alexander Volberg,
\newline{Department of  Mathematics, Michigan State University}
\newline{and the University of Edinburgh}
\newline{{\tt volberg@math.msu.edu}\ and\ {\tt a.volberg@ed.ac.uk}}}

\thanks{
The first author was partially supported by RFBR grant 11-01-00584.
\newline
Research of the second author was supported in part by NSF grants  DMS-0501067.}

\begin{abstract}
The sharp constants in the classical John--Nirenberg inequality are found by using
Bellman function approach.
\end{abstract}

\maketitle

\section{Introduction}

Bellman function method in Harmonic Analysis was introduced by Burkholder for
finding the norm in $L^p$ of the Martingale transform. Later it became clear
that the scope of the method is quite wide. 

After Burkholder the first
systematic application of this technique appeared in 1995 in the first preprint
version of~\cite{NTV1}.  It was vastly developed in~\cite{NT} and in (now)
numerous  papers that followed. A small part of this literature can be found in \cite{NTV} and in lecture notes \cite{VoPr} and in the references section of the present article.
It became clear that magic  Burkholder function from \cite{Bu1} does not have too much in common with
Harmonic Analysis, it is a natural dweller of the area called
Stochastic Optimal Control. It is a solution of a corresponding Bellman
equation (or a dynamic programming equation), which appears from solving optimization problems. It turns out that the point of view that many Harmonic Analysis problems are optimization problems can be profitable. And this is even though many interesting extremal Harmonic Analysis problems may not have an actual extremizer because of the lack of the compactness in the problem.  However many Harmonic Analysis problem have their specific Bellman function, which is a solution of a certain Bellman (usually non-linear) PDE.

A crucial property of elements of $\BMO$-space, the exponential decay
of their distribution function, was established in the classical
paper~\cite{JN}; it is known as the John--Nirenberg inequality.

For an interval $I$, and a real-valued function $\vf\in L^1(I)$, let
$\av{\vf}I$ be the average of $\vf$ over $I$, i.e.,
$$
\av{\vf}I=\frac1{|I|}\int_I\vf,
$$
where $|I|$ stands for Lebesgue measure of $I$. For $1\le p<\infty$, let
\eq[BMO]
{ \BMO(J)=\left\{\vf\in L^1(J): \av{|\vf-\av{\vf}I|^p}I\le
C^p<\infty,\; \forall I\subset J\right\} }
with the best (smallest) such $C$ being the corresponding ``norm'' of
$\vf$. For $\ve\ge0,$ let
$$
\BMO_\ve(J)=\{\vf\in\BMO(J):\|\vf\|\le\ve\}.
$$
The classical definition of John and Nirenberg uses $p=1$; it is known that the
norms are equivalent for different $p$'s. For every $\vf\in\BMO(J)$ and
every $\lambda\in\mathbb{R}$ the classical John--Nirenberg inequality consists
in the following assertion.

\smallskip

\noindent{\bf Theorem}~(John, Nirenberg; weak form) \eq[i1.5]{
\frac1{|J|}|\{s\in J\colon |\vf(s)-\av{\vf}J|\ge\lambda\}|\le
c_1e^{-c_2\lambda/\|\vf\|_{\BMO(J)}}. }

\smallskip

I refer to this statement as to the weak form of  the John--Nirenberg
inequality to distinguish it from the following equivalent assertion.

\smallskip

\noindent{\bf Theorem}~(John, Nirenberg; integral form) {\it There exists
$\ve_0>0$ such that for every $\ve$, $0\le\ve<\ve_0$, there is $C(\ve)>0$ such
that for any function $\vf$, $\vf\in \BMO_{\ve}(J)$}, the following
inequality holds
$$
\av{e^\vf}J\le C(\ve)e^{\av{\vf}J}\,.
$$

\smallskip

The sharp constants in the integral form were found in~\cite{V} and~\cite{SV}.
In the second paper the dyadic analog $\BMO^d$ is considered as well, for which
every subinterval $I$ of $J$ in definition~\eqref{BMO} is an element of the
dyadic lattice rooted in $J$. It appears that the constants in the dyadic case
and the usual one are different.

The mentioned constants were found by using the so called Bellman function method
(see  survey~\cite{NTV} for historical remarks). Namely, the Bellman 
function of the corresponding
extremal problem (the definition see below) was found explicitly. This function
carries all the information about the problem: not only the sharp constants,
but, for example, construction of extremal test functions ({\it extremizers}). The
Bellman function corresponding to the integral John--Nirenberg inequality was
found by solving the boundary value problem for the Bellman equation. In that
case the Bellman equation was a second order PDE with two variables, and due to
a natural homogeneity of the problem, the Bellman PDE was reduced to an
ordinary differential equation, which was successfully solved. The
corresponding Bellman equation for the week John--Nirenberg inequality has an
additional parameter $\lambda$ preventing a similar reducing of the Bellman PDE
to an ordinary differential equation.

The Bellman equations for all these problems are in fact partial cases of the
Monge--Amp\`ere equation. After finding possibility to solve this type of
equation explicitly (see~\cite{SSV}, \cite{VV}) we are able to find the Bellman
function (and therefore, the sharp constants) for the weak John--Nirenberg
inequality as well. And this solution is described in the present paper.

We shall work with $L^2$-based $\BMO$-norm, i.e., $p=2$ will be chosen
in~\eqref{BMO}. For the classical case $p=1$, Korenovskii~\cite{K} established
the exact value $c_2=2/e$ using the equimeasurable rearrangements of the test
function and the ``sunrise lemma''. But to apply the Bellman function method
the $L^2$-based $\BMO$-norm is more appropriate. Some Bellman-type function
(so-called supersolution) for the weak John--Nirenberg inequality was proposed
by Tao in~\cite{T}, where there was no attempt to find true Bellman function
and sharp constants. In the present paper it will be proved that for $p=2$ the
sharp constant are $c_1=\frac4{e^2}$ and $c_2=1$.

\section{Definitions and statements of the main results}
\subsection{Bellman functions}
Now the main subject of the paper will be introduced, the Bellman function
corresponding to the John--Nirenberg inequality. First of all we define the
following set of test functions
\begin{multline}
\label{test}
S_\ve(x)=S(x_1,x_2;\ve)=\\
\{\vf\in\BMO(J)\colon\av\vf J=x_1, \av{\vf^2}J=x_2,
\av{|\vf-\av\vf I|^2}I\le\ve^2\;\forall I\subset J\}\,.
\end{multline}
For any test function $\vf$ the point $x=(x_1,x_2)=(\av{\vf}J,\av{\vf^2}J)$
belongs to the parabolic strip
\eq[Omega]{ \Omega_\ve=\{x=(x_1,x_2)\colon x_1^2\le x_2\le x_1^2+\ve^2\}\,. }
Indeed, the left inequality $x_1^2\le x_2$ is simply
the Cauchy inequality, but the right one $x_2\le
x_1^2+\ve^2$ follows from the fact that $\vf\in\BMO_\ve(J)$:
$$
x_2-x_1^2=\av{\vf^2}J-\av\vf J^2=\av{|\vf-\av\vf J|^2}J\le\ve^2\,.
$$

Now we define the Bellman $\Bel$ function corresponding to the weak
John--Nirenberg inequality:
\eq[Bel]{\Bel(x;\lambda)\df
\Bel(x;\lambda,\ve)\df\frac1{|I|}\sup\big\{|\{s\in I\colon|\vf(s)|\ge\lambda\}|
\colon\vf\in S_\ve(x)\big\}\,.
}

This function is defined on $\Omega$ and it supplies us with the sharp estimate
of the distribution function
\eq[Bel-JN]{
\frac1{|J|}|\{s\in J\colon|\vf(s)-\av\vf J|\ge\lambda\}|\;\;\le
\sup_{\xi\in[0,\ve^2]}\Bel(0,\xi;\lambda)\quad \forall\vf\in\BMO_\ve\,.
}
To check this, we consider a new function $\tilde\vf\df\vf+c$. If $\vf\in
S_\ve(x)$, then $\tilde\vf\in S_\ve(\tilde x)$, where $\tilde x_1=x_1+c$ and
$\tilde x_2=x_2+2cx_1+c^2$. Therefore, by definition~\eqref{Bel}, we have
$$
\frac1{|J|}|\{s\in J\colon|\tilde\vf(s)|\ge\lambda\}|\le\Bel(\tilde x;\lambda)\,.
$$
If we take now $c=-\av\vf J=-x_1$, we get $\tilde x_1=0$, $\tilde x_2=x_2-x_1^2$,
and the latter inequality turns into
$$
\frac1{|J|}|\{s\in J\colon|\vf(s)-\av\vf J|\ge\lambda\}|\le
\Bel(0,\tilde x_2;\lambda)\le
\sup_{\xi\in[0,\ve^2]}\Bel(0,\xi;\lambda)\,.
$$

So, to find the sharp constants in the weak John--Nirenberg inequality we prove
the following theorem.
\begin{theorem}
\label{WJNBel}
For $0\le\lambda\le\ve$ split $\Omega$ in three subdomains {\rm (see~Fig.~\ref{smallLambda}):}
\begin{figure}
\begin{center}
\begin{picture}(320,200)
\thinlines
\put(160,0){\vector(0,1){180}}
\put(10,20){\vector(1,0){300}}
\put(200,26){\vector(-1,1){10}}
\put(120,26){\vector(1,1){10}}
\linethickness{.3pt}
\thicklines
\qbezier[1500](40,164)(160,-124)(280,164)
\qbezier[1500](40,173)(160,-40)(280,173)
\qbezier[1500](160,20)(160,20)(96,61)
\qbezier[1500](160,20)(160,20)(224,61)
\qbezier[1500](96,61)(96,61)(224,61)
\put(148,45){\footnotesize $\Omega_2$}
\put(210,67){\footnotesize $\Omega_1$}
\put(203,23){\footnotesize $\Omega_3^+$}
\put(107,23){\footnotesize $\Omega_3^-$}
\put(280,150){\footnotesize $x_2=x_1^2$}
\put(205,120){\footnotesize $x_2=x_1^2+\ve^2$}
\put(96.2,61){\circle*{2}}
\multiput(96,20)(0,3){14}{\circle*{1}}
\put(96,20){\circle*{2}}
\put(91,8){\footnotesize $-\!\lambda$}
\put(160,20){\circle*{2}}
\put(224,60.7){\circle*{2}}
\multiput(224,20)(0,3){14}{\circle*{1}}
\put(224,20){\circle*{2}}
\put(221,8){\footnotesize $\lambda$}
\end{picture}
\caption{}
\label{smallLambda}
\end{center}
\end{figure}
\begin{align*}
\Omega_1&=\{x\in\Omega\colon x_2\ge\lambda^2\}\,,
\\
\Omega_2&=\{x\in\Omega\colon \lambda|x_1|\le x_2\le\lambda^2\}\,,
\\
\Omega_3&=\{x\in\Omega\colon x_2<\lambda|x_1|\}\,,
\end{align*}
then
\eq[JNBel1]{
\Bel(x;\lambda,\ve)=\begin{cases}
\qquad1\,,& x\in\Omega_1\,,\rule[-15pt]{0pt}{15pt}
\\
\displaystyle
\qquad\frac{x_2}{\lambda^2}\,,& x\in\Omega_2\,,\rule[-15pt]{0pt}{15pt}
\\
\displaystyle
\frac{x_2-x_1^2}{x_2+\lambda^2-2\lambda|x_1|}\,,\qquad & x\in\Omega_3\,.
\end{cases}
}
\smallskip

For $\ve<\lambda\le2\ve$ split $\Omega$ in four subdomains {\rm (see~Fig.~\ref{medLambda}):}
\begin{figure}
\begin{center}
\begin{picture}(320,200)
\thinlines
\put(160,0){\vector(0,1){180}}
\put(10,20){\vector(1,0){300}}
\put(273,114){\vector(-1,0){18}}
\put(213,77){\vector(1,-2){7}}
\put(200,26){\vector(-1,1){10}}
\put(152,43){\vector(1,-3){4}}
\put(48,114){\vector(1,0){18}}
\put(108,77){\vector(-1,-2){7}}
\put(120,26){\vector(1,1){10}}
\linethickness{.3pt}
\thicklines
\qbezier[1500](40,164)(160,-124)(280,164)
\qbezier[1500](40,173)(160,-100)(280,173)
\qbezier[1500](160,20)(160,20)(96,61)
\qbezier[1500](160,20)(160,20)(224,61)
\qbezier[1500](96,61)(96,61)(136,42)
\qbezier[1500](224,61)(224,61)(184,42)
\qbezier[1500](96,61)(96,61)(56,139)
\qbezier[1500](224,61)(224,61)(264,139)
\put(148,45){\footnotesize $\Omega_4$}
\put(276,112){\footnotesize $\Omega_1^+$}
\put(210,80){\footnotesize $\Omega_2^+$}
\put(203,23){\footnotesize $\Omega_3^+$}
\put(33,112){\footnotesize $\Omega_1^-$}
\put(107,80){\footnotesize $\Omega_2^-$}
\put(107,23){\footnotesize $\Omega_3^-$}
\put(280,150){\footnotesize $x_2=x_1^2$}
\put(205,120){\footnotesize $x_2=x_1^2+\ve^2$}
\put(56,140){\circle*{2}}
\multiput(56,20)(0,3){40}{\circle*{1}}
\put(56,20){\circle*{2}}
\put(42,9){\footnotesize $-\ve\!-\!\lambda$}
\put(96,61){\circle*{2}}
\multiput(96,20)(0,3){14}{\circle*{1}}
\put(96,20){\circle*{2}}
\put(91,8){\footnotesize $-\!\lambda$}
\put(136,42){\circle*{2}}
\multiput(136,20)(0,3){8}{\circle*{1}}
\put(136,20){\circle*{2}}
\put(128,8){\footnotesize $\ve\!-\!\lambda$}
\put(160,20){\circle*{2}}
\put(184,42){\circle*{2}}
\multiput(184,20)(0,3){8}{\circle*{1}}
\put(184,20){\circle*{2}}
\put(184,8){\footnotesize $\lambda\!-\!\ve$}
\put(224,60.7){\circle*{2}}
\multiput(224,20)(0,3){14}{\circle*{1}}
\put(224,20){\circle*{2}}
\put(221,8){\footnotesize $\lambda$}
\put(264,138){\circle*{2}}
\multiput(264,20)(0,3){40}{\circle*{1}}
\put(264,20){\circle*{2}}
\put(256,8){\footnotesize $\lambda\!+\!\ve$}
\end{picture}
\caption{}
\label{medLambda}
\end{center}
\end{figure}
\begin{align*}
\Omega_1&=\{x\in\Omega\colon |x_1|\ge\lambda \text{ and }
x_2\le2(\lambda+\ve)|x_1|-\lambda^2-2\ve\lambda
\text{ for }|x_1|<\lambda+\ve,\}\,,
\\
\Omega_2&=\{x\in\Omega\colon \lambda-\ve\le|x_1|\le\lambda+\ve,\,
x_2\ge\max\{2\lambda |x_1|-\lambda^2\pm2\ve(|x_1|-\lambda)\}\}\,,
\\
\Omega_3&=\{x\in\Omega\colon x_2<\lambda|x_1|\}\,,
\\
\Omega_4&=\{x\in\Omega\colon x_2\ge\lambda|x_1|\text{ and }
x_2\le2(\lambda-\ve)|x_1|-\lambda^2+2\ve\lambda\text{ for }|x_1|>\lambda-\ve\}\,,
\end{align*}
then
\eq[JNBel2]{
\Bel(x;\lambda,\ve)=\begin{cases}
\qquad\qquad1\,,& x\in\Omega_1\,,\rule[-15pt]{0pt}{15pt}
\\
\displaystyle
\frac{2(\lambda^2-\ve^2)|x_1|-(\lambda-\ve)x_2+\lambda(2\ve^2+\ve\lambda-\lambda^2)}
{2\ve\lambda^2}\,,& x\in\Omega_2\,,\rule[-15pt]{0pt}{15pt}
\\
\displaystyle
\qquad\frac{x_2-x_1^2}{x_2+\lambda^2-2\lambda|x_1|}\,,\qquad &
x\in\Omega_3\,,\rule[-20pt]{0pt}{15pt}
\\
\displaystyle
\qquad\qquad\frac{x_2}{\lambda^2}\,,& x\in\Omega_4\,.
\end{cases}
}

\smallskip

For $\lambda>2\ve$ split $\Omega$ in five subdomains {\rm (see~Fig.~\ref{bigLambda}):}
\begin{figure}
\begin{center}
\begin{picture}(320,200)
\thinlines
\put(160,0){\vector(0,1){180}}
\put(10,20){\vector(1,0){300}}
\put(294,174){\vector(-2,-1){14}}
\put(240,105){\vector(2,1){14}}
\put(243,55){\vector(-1,1){15}}
\put(181,44){\vector(1,-3){4}}
\put(154,38){\vector(1,-4){3}}
\put(26,174){\vector(2,-1){14}}
\put(81,105){\vector(-2,1){14}}
\put(79,58){\vector(1,1){13}}
\put(133,44){\vector(1,-3){5}}
\linethickness{.3pt}
\thicklines
\qbezier[1500](40,164)(160,-124)(280,164)
\qbezier[1500](40,173)(160,-115)(280,173)
\qbezier[1500](160,20)(160,20)(188,37)
\qbezier[1500](160,20)(160,20)(132,37)
\qbezier[1500](195,32)(195,32)(254,108)
\qbezier[1500](254,108)(254,108)(286,188)
\qbezier[1500](125.5,32)(125.5,32)(66,108)
\qbezier[1500](66,108)(66,108)(34,188)
\put(150,39){\footnotesize $\Omega_5$}
\put(296,172){\footnotesize $\Omega_1^+$}
\put(225,100){\footnotesize $\Omega_2^+$}
\put(240,45){\footnotesize $\Omega_3^+$}
\put(178,47){\footnotesize $\Omega_4^+$}
\put(14,172){\footnotesize $\Omega_1^-$}
\put(83,100){\footnotesize $\Omega_2^-$}
\put(70,50){\footnotesize $\Omega_3^-$}
\put(128,47){\footnotesize $\Omega_4^-$}
\put(280,150){\footnotesize $x_2=x_1^2$}
\put(205,120){\footnotesize $x_2=x_1^2+\ve^2$}
\put(36.8,182){\circle*{2}}
\multiput(36.8,20)(0,3){55}{\circle*{1}}
\put(36.8,20){\circle*{2}}
\put(22,12){\footnotesize $-\ve\!-\!\lambda$}
\put(66.3,108.3){\circle*{2}}
\multiput(66.3,20)(0,3){30}{\circle*{1}}
\put(66.3,20){\circle*{2}}
\put(62,6){\footnotesize $-\!\lambda$}
\put(96,70){\circle*{2}}
\multiput(96,20)(0,3){17}{\circle*{1}}
\put(96,20){\circle*{2}}
\put(88,12){\footnotesize $\ve\!-\!\lambda$}
\put(126,32){\circle*{2}}
\multiput(126,20)(0,3){5}{\circle*{1}}
\put(126,20){\circle*{2}}
\put(110,6){\footnotesize $2\ve\!-\!\lambda$}
\put(132,37){\circle*{2}}
\multiput(132,20)(0,3){6}{\circle*{1}}
\put(132,20){\circle*{2}}
\put(129,12){\footnotesize $-\ve$}
\put(160,20){\circle*{2}}
\put(188,36.5){\circle*{2}}
\put(188,20){\circle*{2}}
\multiput(188,20)(0,3){6}{\circle*{1}}
\put(182,12){\footnotesize $\ve$}
\put(195,32){\circle*{2}}
\multiput(195,20)(0,3){5}{\circle*{1}}
\put(195,20){\circle*{2}}
\put(190,6){\footnotesize $\lambda\!-\!2\ve$}
\put(224.7,70){\circle*{2}}
\multiput(224.7,20)(0,3){17}{\circle*{1}}
\put(224.7,20){\circle*{2}}
\put(218,12){\footnotesize $\lambda\!-\!\ve$}
\put(254,107.5){\circle*{2}}
\multiput(254,20)(0,3){30}{\circle*{1}}
\put(254,20){\circle*{2}}
\put(251,6){\footnotesize $\lambda$}
\put(284,182){\circle*{2}}
\multiput(284,20)(0,3){55}{\circle*{1}}
\put(284,20){\circle*{2}}
\put(276,12){\footnotesize $\lambda\!+\!\ve$}
\end{picture}
\caption{}
\label{bigLambda}
\end{center}
\end{figure}
\begin{align*}
\Omega_1&=\{x\in\Omega\colon |x_1|\ge\lambda \text{ and }
x_2\le2(\lambda+\ve)|x_1|-\lambda^2-2\ve\lambda
\text{ for }|x_1|<\lambda+\ve,\}\,,
\\
\Omega_2&=\{x\in\Omega\colon \lambda-\ve\le|x_1|\le\lambda+\ve,\,
x_2\ge\max\{2\lambda |x_1|-\lambda^2\pm2\ve(|x_1|-\lambda)\}\}\,,
\\
\Omega_3&=\{x\in\Omega\colon x_2<2(\lambda-\ve)|x_1|-\lambda^2+2\ve\lambda\}\,,
\\
\Omega_4&=\{x\in\Omega\colon x_2\ge2(\lambda-\ve)|x_1|-\lambda^2+2\ve\lambda
\text{ and }x_2\le2\ve|x_1|\text{ for }|x_1|<\ve\}\,,
\\
\Omega_5&=\{x\in\Omega\colon x_2\ge2\ve|x_1|\}\,,
\end{align*}
then
\eq[JNBel3]{
\Bel(x;\lambda,\ve)=\begin{cases}
\qquad\qquad1\,,& x\in\Omega_1\,,\rule[-15pt]{0pt}{15pt}
\\
\displaystyle
\qquad1-\frac{x_2-2(\lambda+\ve)|x_1|+\lambda^2+2\ve\lambda}
{8\ve^2}\,,& x\in\Omega_2\,,\rule[-15pt]{0pt}{15pt}
\\
\displaystyle
\qquad\frac{x_2-x_1^2}{x_2+\lambda^2-2\lambda|x_1|}\,,\qquad &
x\in\Omega_3\,,\rule[-20pt]{0pt}{15pt}
\\
\displaystyle
\frac e2\left(\!1\!-\!\sqrt{1\!-\!\frac{x_2\!-\!x_1^2}{\ve^2}}\right)
\exp\left\{\frac{|x_1|\!-\!\lambda}\ve\!+\!\sqrt{1\!-\!\frac{x_2\!-\!x_1^2}{\ve^2}}\right\},
& x\in\Omega_4\,,\rule[-20pt]{0pt}{15pt}
\\
\displaystyle
\qquad\qquad\frac{x_2}{4\ve^2}\exp\left\{2-\frac\lambda\ve\right\},& x\in\Omega_5\,.
\end{cases}
}
\end{theorem}

\begin{cor}
If $\vf\in\BMO_\ve(I)$, then
$$
\frac1{|I|}|\{s\in I\colon|\vf(s)-\av\vf I|\ge\lambda\}|\;\;\le
\begin{cases}
1,&\text{if }\quad0\le\lambda\le\ve,\rule[-15pt]{0pt}{15pt}
\\
\displaystyle
\frac{\ve^2}{\lambda^2}&\text{if}\quad\ve\le\lambda\le2\ve,\rule[-15pt]{0pt}{15pt}
\\
\displaystyle
\frac{e^2}4e^{-\lambda/\ve}\quad&\text{if}\quad2\ve\le\lambda,
\end{cases}
$$
and this bound is sharp.
\end{cor}

\begin{proof}
According to formula~\eqref{Bel-JN} it is sufficient to calculate
$$
\sup_{\xi\in[0,\ve^2]}\Bel(0,\xi;\lambda,\ve)\,.
$$
Since $\Bel(0,x_2;\lambda,\ve)$ is an increasing function in $x_2$,
this supremum is just the value $\Bel(0,\ve^2;\lambda,\ve)$, what
yields the stated formula.
\end{proof}

Before we start to prove Theorem~\ref{WJNBel}, where the Bellman
function has two singularities on the boundary at the points
$x=(\pm\lambda,\lambda^2)$, let us consider the simplest possible
extremal problem with one singularity. We shall consider two
extremal problems simultaneously: one estimate from above
and the second estimate from below. So, we define two Bellman
functions: $\Bel_{\max}$ and $\Bel_{\min}$.
$$
\Bel_{\max}(x;\lambda,\ve)\df\frac1{|I|}\sup\big\{|\{s\in I\colon\vf(s)\ge\lambda\}|
\colon\vf\in S_\ve(x)\big\}\,,
$$
$$
\Bel_{\min}(x;\lambda,\ve)\df\frac1{|I|}\inf\big\{|\{s\in I\colon\vf(s)\ge\lambda\}|
\colon\vf\in S_\ve(x)\big\}\,,
$$
For these function the following formula will be proved:

\begin{theorem}
\label{1jump}
Split $\Omega$ in the following five subdomains {\rm (see~Fig.~\ref{oneLambda}):}
\begin{figure}
\begin{center}
\begin{picture}(320,200)
\thinlines
\put(160,0){\vector(0,1){180}}
\put(10,20){\vector(1,0){300}}
\linethickness{.7pt}
\qbezier[1500](30,190)(160,-150)(290,190)
\qbezier[1500](40,195)(160,-93)(280,195)
\qbezier[1500](175,22)(175,22)(284,176)
\qbezier[1500](175,22)(175,22)(64.8,112)
\put(170,42){\footnotesize $\Omega_3$}
\put(220,73){\footnotesize $\Omega_2$}
\put(276.5,180){\footnotesize $\Omega_1$}
\put(100,63){\footnotesize $\Omega_4$}
\put(53,140){\footnotesize $\Omega_5$}
\put(25,122){\footnotesize $x_2=x_1^2$}
\put(205,140){\footnotesize $x_2=x_1^2+\ve^2$}
\put(64.2,112.5){\circle*{2}}
\multiput(64.2,20)(0,3){32}{\circle*{1}}
\put(64.2,20){\circle*{2}}
\put(54,8){\footnotesize $\lambda\!-\!2\ve$}
\put(120,67.2){\circle*{2}}
\multiput(120,20)(0,3){16}{\circle*{1}}
\put(120,20){\circle*{2}}
\put(111,8){\footnotesize $\lambda\!-\!\ve$}
\put(175.3,20){\circle*{2}}
\put(172,8){\footnotesize $\lambda$}
\put(175.3,22.3){\circle*{2}}
\put(232,102){\circle*{2}}
\multiput(232,20)(0,3){28}{\circle*{1}}
\put(232,20){\circle*{2}}
\put(222,8){\footnotesize $\lambda\!+\!\ve$}
\put(285,176.5){\circle*{2}}
\multiput(285,20)(0,3){53}{\circle*{1}}
\put(285,20){\circle*{2}}
\put(275,8){\footnotesize $\lambda\!+\!2\ve$}
\end{picture}
\caption{}
\label{oneLambda}
\end{center}
\end{figure}
\begin{align*}
\Omega_1&=\{x\in\Omega\colon x_1\ge\lambda+\ve,\,
x_2\ge2(\lambda+\ve)x_1-\lambda^2-2\ve\lambda\}\,,
\\
\Omega_2&=\{x\in\Omega\colon x_2\le2(\lambda+\ve)x_1-\lambda^2-2\ve\lambda\}\,,
\\
\Omega_3&=\{x\in\Omega\colon \lambda-\ve\le x_1\le\lambda+\ve,\,
x_2\ge2\lambda x_1-\lambda^2+2\ve|x_1-\lambda|\}\,,
\\
\Omega_4&=\{x\in\Omega\colon x_2\le2(\lambda-\ve)x_1-\lambda^2+2\ve\lambda\}\,,
\\
\Omega_5&=\{x\in\Omega\colon x_1\le\lambda-\ve,\,
x_2\ge2(\lambda-\ve)x_1-\lambda^2+2\ve\lambda\}\,.
\end{align*}
Then
\eq[1max]{
\Bel_{\max}(x;\lambda,\ve)=
\begin{cases}
\qquad\qquad1\,,& \hskip-18pt x\in\Omega_1\cup\Omega_2\,,\rule[-15pt]{0pt}{15pt}
\\
\displaystyle
\qquad1-\frac{x_2-2(\lambda+\ve)x_1+\lambda^2+2\ve\lambda}
{8\ve^2}\,,& x\in\Omega_3\,,\rule[-15pt]{0pt}{15pt}
\\
\displaystyle
\qquad\frac{x_2-x_1^2}{x_2+\lambda^2-2\lambda x_1}\,,\qquad &
x\in\Omega_4\,,\rule[-20pt]{0pt}{15pt}
\\
\displaystyle
\frac e2\left(\!1\!-\!\sqrt{1\!-\!\frac{x_2\!-\!x_1^2}{\ve^2}}\right)
\exp\left\{\frac{x_1\!-\!\lambda}\ve\!+\!\sqrt{1\!-\!\frac{x_2\!-\!x_1^2}{\ve^2}}\right\},
& x\in\Omega_5\,,
\end{cases}
}
and
\eq[1min]{
\Bel_{\min}(x;\lambda,\ve)=
\begin{cases}
\qquad\qquad0\,,& \hskip-18pt x\in\Omega_5\cup\Omega_4\,,\rule[-15pt]{0pt}{15pt}
\\
\displaystyle
\frac{x_2-2(\lambda-\ve)x_1+\lambda^2-2\ve\lambda}
{8\ve^2}\,,& x\in\Omega_3\,,\rule[-15pt]{0pt}{15pt}
\\
\displaystyle
\qquad1-\frac{x_2-x_1^2}{x_2+\lambda^2-2\lambda x_1}\,,\qquad &
x\in\Omega_2\,,\rule[-20pt]{0pt}{15pt}
\\
\displaystyle
1-\frac e2\left(\!1\!-\!\sqrt{1\!-\!\frac{x_2\!-\!x_1^2}{\ve^2}}\right)
\exp\left\{\frac{\lambda\!-\!x_1}\ve\!+\!\sqrt{1\!-\!\frac{x_2\!-\!x_1^2}{\ve^2}}\right\},
& x\in\Omega_1\,.
\end{cases}
}
\end{theorem}

\section{Proofs of the theorems}

Let us show that it is sufficient to prove Theorem~\ref{1jump} only for $\Bel_{\max}$,
then we get the lower Bellman function automatically. Indeed, since $\Bel_{\max}$ is
a continuous function in $\lambda$ for any fixed $x$ except one point on the lower boundary (i.e.
$x_2>x_1^2$), for any such $x$ and any $\eta>0$ we have:
$$
|\{s\in I\colon\vf(s)\ge\lambda+\eta\}|
\le|\{s\in I\colon\vf(s)>\lambda\}|
\le|\{s\in I\colon\vf(s)\ge\lambda\}|\,,
$$
which yields
$$
\Bel(x;\lambda+\eta)
\le\sup\big\{|\{s\in I\colon\vf(s)>\lambda\}|
\colon\vf\in S_\ve(x)\big\}
\le\Bel(x;\lambda)\,.
$$
Therefore, the Bellman function for the strict inequality in
the definition is the same as the Bellman function for the
non strict inequality, except one point on the boundary $x=(\lambda,\lambda^2)$, where we know the Bellman function
from the beginning, because for the points of the lower boundary
the set $S_\ve(x)$ consists of  only the constant test function
$\vf=x_1=\lambda$.

At the point $x=(\lambda,\lambda^2)$, where both Bellman function
are equal to $1$, $\Bel_{\max}(x)=\Bel_{\min}(x)=1$. At all other
points we have the following relation
$$
\Bel_{\min}(x_1,x_2;\lambda)=1-\Bel_{\max}(-x_1,x_2;-\lambda).
$$
Indeed,
\begin{align*}
\Bel_{\min}(x_1,x_2;\lambda)&=\frac1{|J|}\inf\big\{|\{s\in J\colon\vf(s)\ge\lambda\}|
\colon\vf\in S_\ve(x)\big\}
\\
&=1-\frac1{|J|}\sup\big\{|\{s\in J\colon\vf(s)<\lambda\}|\colon\vf\in S_\ve(x)\big\}
\\
&=1-\frac1{|J|}\sup\big\{|\{s\in J\colon-\vf(s)>-\lambda\}|\colon\vf\in S_\ve(x)\big\}
\\
&=1-\Bel_{\max}(-x_1,x_2;-\lambda)\,.
\end{align*}
Using this relation we obtain~\eqref{1min} from~\eqref{1max}.

When proving Theorem~\ref{WJNBel} we denote by $B$ the function from the right-hand
side of either~\eqref{JNBel1}, or~\eqref{JNBel2}, or~\eqref{JNBel3}, depending on
the relation between $\lambda$ and $\ve$, and $B$ will be the function from the
right-hand side of~\eqref{1max} in the proof of Theorem~\ref{1jump}. In any case $B$
will be a candidate for the role of the Bellman function, and to prove the
theorem we need in each case to check two inequalities for the corresponding
pair $\Bel$ and $B$: $\Bel(x)\le B(x)$ and $\Bel(x)\ge B(x)$ for every point
$x\in\Omega_\ve$.

To prove the upper estimate, we need, first, the local concavity of the function $B$:
\begin{equation}
\label{main}
B(\alpha_+x^++\alpha_-x^-)\ge \alpha_+B(x^+)+\alpha_-(x^-)\,,
\quad\alpha_\pm>0,\ \alpha_++\alpha_-=1,
\end{equation}
for any pair $x^\pm\in\Omega_\ve$ such that the whole straight-line segment
$[x^-,x^+]$ is in $\Omega_\ve$, and, second, the following splitting lemma
that can be found in~\cite{V} or~\cite{SV}:

\begin{lemma}[Splitting lemma]
\label{splitting}
Fix two positive numbers $\ve,\delta,$ with $\ve<\delta.$ For an
arbitrary interval $I$ and any function $\vf\in\BMO_\ve(I),$
there exists a splitting $I=I_+\cup I_-$ such that the whole
straight-line segment $[x^{I_-},x^{I_+}]$ is inside $\Omega_\delta$.
Moreover\textup, the parameters of splitting $\alpha_\pm\df|I_{\pm}|/|I|$
are separated form $0$ and $1$ by constants depending on $\ve$ and
$\delta$ only\textup, i.e. uniformly with respect to the choice of
$I$ and $\vf.$
\end{lemma}

Here the following notation was used: for a function $\vf\in\BMO_\ve(J)$
and a subinterval $I\subset J$ we define a {\it Bellman point}
$x^I\df(\av\vf I,\av{\vf^2}I)$ in the domain $\Omega_\ve$.

Using this lemma we prove the following result.

\begin{lemma}
\label{Bell_induction}
Let $G$ be a locally concave bounded function on $\Omega_\delta$\textup,
$\delta>\ve$\textup, and $E$ is a measurable subset of $\R$. If the function $G$ satisfies
the following boundary condition
\eq[bcWJN]{
G(x_1,x_1^2)=
\begin{cases}
1,&\text{ if }x_1\in E;
\\
0,&\text{ if }x_1\not\in E,
\end{cases}
}
then
$$
\frac1{|I|}|\{s\colon\vf(s)\in E\}|\le G(x)
$$
for all $\vf\in S_\ve(x)$.
\end{lemma}

We shall use this lemma to prove the theorem putting
$G(x)=B(x;\lambda,\delta)$ and then, using continuity of $B(x;\lambda, \delta)$ in $\delta$,
we pass to the limit $\delta\to\ve$. In such a way we get the upper estimate
$$
\Bel(x;\lambda,\ve)\le B(x;\lambda,\ve).
$$

\begin{proof}[Proof of Lemma~\ref{Bell_induction}]
Procedure of the proof is standard, as
in~\cite{V} or~\cite{SV}: we apply repeatedly main inequality~\eqref{main}
each time splitting the interval according to Lemma~\ref{splitting}.

Fix a function $\vf\in S_\ve(x)$. By the splitting lemma we
can split every subinterval $I\subset J,$ in such a way that the
segment $[x^{I_-},x^{I_+}]$ is inside $\Omega_\delta$. Since $G$
is locally concave, we have (we drop temporarily parameter $\delta$)
$$
|I|G(x^I)\ge|I_+|G(x^{I_+})+|I_-|G(x^{I_-})
$$
for any such splitting. Repeating this procedure $n$ times we get $2^n$ subintervals
of $n$-th generation (this set of intervals we denote by $\D_n$). So, we can write
the following chain of inequalities:
$$
|J|G(x^J)\ge|J_+|G(x^{J_+})+|J_-|G(x^{J_-})
\ge\sum_{I\in\D_n}|I|G(x^I)=\int_JG(x^{(n)}(s))\,ds\,,
$$
where $x^{(n)}(s)=x^I,$ when $s\in I,$ $I\in\D_n.$ By the Lebesgue
differentiation theorem we have $x^{(n)}(s)\to(\vf(s),\vf^2(s))$
almost everywhere. (We have used here the fact that we split the
intervals so that all coefficients $\alpha_\pm$ are uniformly separated
from~$0$ and~$1,$ and, therefore, $\max\{|I|\colon I\in\D_n\}\to0$ as
$n\to\infty$.) Since $G$ is bounded, we can pass to the limit in
this inequality by the Lebesgue dominated convergence theorem. Using
the boundary condition~\eqref{bcWJN} we obtain:
$$
|J|G(x^J)\ge\int_JG(\vf(s),\vf^2(s))\,ds=
\int_{\{s\colon\vf(s)\in E\}}\hskip-30pt ds\quad=|\{s\colon\vf(s)\in E\}|\,.
$$
Dividing the obtained inequality by $|J|$,
we come to the desired inequality.
\end{proof}

To complete proving the upper estimate $\Bel\le B$ both in Theorems~\ref{WJNBel}
and~\ref{1jump} we need to check local concavity of the functions $B$ defined
by~\eqref{JNBel1}, \eqref{JNBel2}, \eqref{JNBel3}, and~\eqref{1max}.

Let us check the most difficult case~\eqref{JNBel3}. In all other cases the
consideration is analogous.
$$
\frac{\partial B}{\partial x_1}=\begin{cases}
\qquad\qquad0\,,& x\in\Omega_1\,,\rule[-15pt]{0pt}{15pt}
\\
\displaystyle
\qquad\frac{\lambda+\ve}{4\ve^2}\sign{x_1}\,,&
x\in\Omega_2\,,\rule[-15pt]{0pt}{15pt}
\\
\displaystyle
\qquad\frac{2(x_2-\lambda|x_1|)(\lambda-|x_1|)}{(x_2+\lambda^2-2\lambda|x_1|)^2}\sign{x_1}\,,
\qquad &x\in\Omega_3\,,\rule[-20pt]{0pt}{15pt}
\\
\displaystyle
\frac e2\cdot\frac{\!\ve\!-\!|x_1|\!-\!\sqrt{\ve^2\!-\!x_2\!+\!x_1^2}}{\ve^2}
\exp\left\{\frac{|x_1|\!-\!\lambda}\ve\!+\!\sqrt{1\!-
\!\frac{x_2\!-\!x_1^2}{\ve^2}}\right\}\sign{x_1},
& x\in\Omega_4\,,\rule[-20pt]{0pt}{15pt}
\\
\displaystyle
\qquad\qquad0,& x\in\Omega_5\,;
\end{cases}
$$
\eq[Bx2]{
\frac{\partial B}{\partial x_2}=\begin{cases}
\qquad\qquad0\,,& x\in\Omega_1\,,\rule[-15pt]{0pt}{15pt}
\\
\displaystyle
\qquad-\frac1{8\ve^2}\,,& x\in\Omega_2\,,\rule[-15pt]{0pt}{15pt}
\\
\displaystyle
\qquad\Big(\frac{|x_1|-\lambda}{x_2+\lambda^2-2\lambda|x_1|}\Big)^2\,,\qquad &
x\in\Omega_3\,,\rule[-20pt]{0pt}{15pt}
\\
\displaystyle
\frac e{4\ve^2}
\exp\left\{\frac{|x_1|\!-\!\lambda}\ve\!+\!\sqrt{1\!-\!\frac{x_2\!-\!x_1^2}{\ve^2}}\right\},
& x\in\Omega_4\,,\rule[-20pt]{0pt}{15pt}
\\
\displaystyle
\qquad\qquad\frac1{4\ve^2}\exp\left\{2-\frac\lambda\ve\right\},& x\in\Omega_5\,.
\end{cases}
}
We see that the function $B$ is $C^1$-smooth on the boundaries $\Omega_5\cap\Omega_4$,
where
$$
B_{x_1}=0\,,\qquad B_{x_2}=\frac1{4\ve^2}\exp\big\{2-\frac\lambda\ve\big\}\,,
$$
and on $\Omega_4\cap\Omega_3$, where
$$
B_{x_1}=-\frac{\lambda-2\ve}{2\ve^2}\,,\qquad B_{x_2}=\frac1{4\ve^2}\,.
$$
On the boundary of $\Omega_2$ the first derivatives have jumps of the needed signs to
keep concavity of $B$.  First of all, we note that it is sufficient to consider a jump
along any direction transversal to the boundary, because along the boundary our functions
coincide and their derivatives coincide as well. (By the way, to check $C^1$-smoothness
of $B$ on the boundary of $\Omega_4$, it was sufficient to verify the continuity of
any partial derivatives, another one would be continuous automatically.) We check
the value of jumps of $B_{x_2}$, because this direction is transversal to the boundary
for any $\ve$. According to~\eqref{Bx2}, on $\Omega_2$ the derivative $B_{x_2}$ is
strictly negative and on $\Omega_1$ and $\Omega_3$ it is nonnegative, therefore $B_{x_2}$
monotonously decreases in $x_2$, as we need. To prove local concavity of $B$ everywhere,
it remains to check that the Hessian matrix
$$
\frac{d^2B}{dx^2}\df\left(\begin{matrix}
B_{x_1x_1}&B_{x_1x_2}\\B_{x_2x_1}&B_{x_2x_2}
\end{matrix}\right)
$$
is non-positive. On $\Omega_1\cup\Omega_2\cup\Omega_5$ the function is linear, and
therefore there is nothing to check. On $\Omega_3$ we have
$$
\frac{d^2B}{dx^2}=\left(\begin{matrix}
\displaystyle
-\frac{2(\lambda^2-x_2)^2}{(x_2+\lambda^2-2\lambda|x_1|)^3}&\rule[-20pt]{0pt}{20pt}
\displaystyle
\frac{2(\lambda^2-x_2)(\lambda-|x_1|)}{(x_2+\lambda^2-2\lambda|x_1|)^3}\sign{x_1}
\\
\displaystyle
\frac{2(\lambda^2-x_2)(\lambda-|x_1|)}{(x_2+\lambda^2-2\lambda|x_1|)^3}\sign{x_1}&
\displaystyle
-\frac{2(\lambda-|x_1|)^2}{(x_2+\lambda^2-2\lambda|x_1|)^3}
\end{matrix}\right)\le0\,,
$$
and on $\Omega_4$
$$
\frac{d^2B}{dx^2}=\frac{e^{1+r-\frac\lambda\ve}}{8\ve^3\sqrt{\ve^2-x_2+x_1^2}}
\left(\begin{matrix}
-4\ve^2r^2&2\ve r
\\
2\ve r&-1\rule{0pt}{15pt}
\end{matrix}\right)\le0\,,
$$
where $r=\frac1\ve\big(|x_1|+\sqrt{\ve^2-x_2+x_1^2}\,\big)$.

In a similar way it is possible to check local concavity of the functions $B$
defined by~\eqref{JNBel1}, \eqref{JNBel2}, and~\eqref{1max}, thus to complete the
proof the upper estimate $\Bel\le B$ both in Theorems~\ref{WJNBel} and~\ref{1jump}.

To prove the converse inequality we construct extremal test functions
({\it extremizers\/}) realizing supremum in the definition of the Bellman function.
Again, we restrict ourself by the consideration of the most difficult 
case~\eqref{JNBel3} only. Moreover, it is sufficient to consider
only the points with $x_1\ge0$, because if $f$ is an extremizer
for a point $(x_1,x_2)$, then the function $-f$ is an extremizer for
the point $(-x_1,x_2)$.

All points of $\Omega_1$ can be represented as a convex combination
of the points of the boundary, where $|x_1|\ge\lambda$, i.e. $B(x)\ge1$. 
Therefore, the corresponding extremizer can be constructed as a step function 
consisting of two constants. Namely, for an arbitrary $x\in\Omega_1$ we draw 
the tangent line to the upper boundary so that
the tangent point is to the right from $x$. First coordinates of two points 
of intersection of that tangent line with the lower boundary are
$u^\pm=x_1\pm\ve+\sqrt{\ve^2-x_2+x_1^2}$, and the corresponding extremizer is
$$
\vf(t)=
\begin{cases}
u^-,&\text{if}\ 0<t<\frac{u^+-x_1}{2\ve},
\\
u^+,&\text{if}\ \frac{u^+-x_1}{2\ve}<t<1.
\end{cases}
$$
By direct calculation we check that $(\av\vf{[0,1]},\av{\vf^2}{[0,1]})=x$
and $\vf\ge\lambda$. First of all we note that
\begin{align*}
u^+-u^-&=2\ve,
\\
u^++u^-&=2\Big(x_1+\sqrt{\ve^2-x_2+x_1^2}\Big),
\\
u^+u^-&=\Big(x_1+\sqrt{\ve^2-x_2+x_1^2}\Big)^2-\ve^2.
\end{align*}
Therefore,
\begin{align*}
\av\vf{[0,1]}&=u^-\frac{u^+-x_1}{2\ve}+u^+\Big(1-\frac{u^+-x_1}{2\ve}\Big)
\\
&=u^+-\frac{(u^+-x_1)(u^+-u^-)}{2\ve}=x_1,
\\
\av{\vf^2}{[0,1]}&=(u^-)^2\frac{u^+-x_1}{2\ve}+(u^+)^2\Big(1-\frac{u^+-x_1}{2\ve}\Big)
\\
&=(u^+)^2-\frac{(u^+-x_1)(u^+-u^-)(u^++u^-)}{2\ve}
=x_1(u^++u^-)-u^+u^-
\\
&=2x_1\Big(x_1+\sqrt{\ve^2-x_2+x_1^2}\Big)
-\Big(x_1+\sqrt{\ve^2-x_2+x_1^2}\Big)^2+\ve^2=x_2.
\end{align*}
To prove that $\vf\ge\lambda$ we need to check that $u_-\ge\lambda$.
If $x_1\ge\lambda+\ve$, then everything is trivial:
$$
u_-\ge x_1-\ve\ge\lambda.
$$
If $x_1<\lambda+\ve$, then the second coordinate of a point $x$ from
$\Omega_1^+$ satisfies the following additional condition
$x_2\le2(\lambda+\ve)x_1-\lambda^2-2\ve\lambda$. Therefore,
$$
\ve^2-x_2+x_1^2\ge\ve^2+x_1^2-2(\lambda+\ve)x_1+\lambda^2+2\ve\lambda
=(\lambda+\ve-x_1)^2,
$$
and hence,
$$
u_-\ge x_1-\ve+|\lambda+\ve-x_1|=\lambda.
$$

What we need more to check is the fact that the $\BMO$-norm of
our extremizer does not exceed $\ve$. In fact it is equal to
$\ve$, since the $\BMO$-norm of any step function consisting of
two steps is equal to the half of the jump and in our case
$u^+-u^-=2\ve$. So, we have proved that $\Bel\ge1$ in $\Omega_1$.

Now, we consider a point $x$ from $\Omega_3^+$.
A similar step function consisting of two steps will be an extremizer
here. We have to draw a straight line through the points $x$ and
$(\lambda,\lambda^2)$. It intersects the lower boundary in one
more point with the first coordinate
$u=\frac{\lambda x_1-x_2}{\lambda-x_1}$. We take a step function
consisting of steps $\lambda$ and $u$:
$$
\vf(t)=
\begin{cases}
\lambda,&\text{if}\ 0<t<a,
\\
u,&\text{if}\ a<t<1,
\end{cases}
$$
where $a=\frac{x_2-x_1^2}{x_2+\lambda^2-2\lambda x_1}$. By
direct calculation we can check that
\begin{align*}
\av\vf{[0,1]}&=\lambda a+u(1-a)=x_1,
\\
\av{\vf^2}{[0,1]}&=\lambda^2 a+u^2(1-a)=x_2.
\end{align*}
The fact that $\vf\in\BMO_\ve$ is geometrically clear, because
a Bellman point corresponding to $\vf$ and any subinterval of $[0,1]$
is in $\Omega_3$. However this is easy to check formally as well.
The jump is
$$
\lambda-u=\lambda-x_1+\frac{x_2-x_1^2}{\lambda-x_1}\,.
$$
Since $x_2\le2(\lambda-\ve)x_1-\lambda^2+2\lambda\ve$ for $x\in\Omega_3^+$, we have
$$
x_2-x_1^2\le(\lambda-x_1)(2\ve-\lambda+x_1),
$$
and hence
$\lambda-u\le2\ve$. So, we conclude that
$$
\Bel\ge a=\frac{x_2-x_1^2}{x_2+\lambda^2-2\lambda x_1}\,.
$$

To consider a point $x\in\Omega_2^+$ we note that this point is
a convex combination of three point on the lower boundary $\Lambda$
and $\Lambda^\pm$ with the first coordinates $\lambda$ and
$\lambda\pm2\ve$ respectively. As a result we construct an
extremizer as a step function consisting of these three steps:
$$
\vf(t)=
\begin{cases}
\lambda-2\ve,&\text{if}\ 0<t<a,
\\
\quad\lambda,&\text{if}\ a<t<b,
\\
\lambda+2\ve,&\text{if}\ b<t<1.
\end{cases}
$$
For $\vf$ to be a test function corresponding the point $x$ (i.e.
for $\av\vf{[0,1]}=x_1$ and $\av{\vf^2}{[0,1]}=x_2$) we need to
take
$$
a=\frac{x_2+\lambda^2-2\lambda x_1-2\ve(x_1-\lambda)}{8\ve^2}
$$
and
$$
b=1-\frac{x_2+\lambda^2-2\lambda x_1+2\ve(x_1-\lambda)}{8\ve^2}\,.
$$
The easiest way to prove that $\vf\in\BMO$ is the following geometric
consideration. Take any straight line, say $L$, passing through $x$
and not intersecting the upper parabola. Note that we need to consider
the oscillation of $\vf$ only over intervals $[\alpha,\beta]$
containing $[a,b]$, because in other case $\vf$ would have on
$[\alpha,\beta]$ only one jump of size $2\ve$, but as we know the
$\BMO$-norm of such step function is just $\ve$. Our point $x$
is a convex combination of three Bellman points $x^{[0,\alpha]}$,
$x^{[\alpha,\beta]}$, and $x^{[\beta,1]}$. But since the points
$x^{[0,\alpha]}=\Lambda^-$ and $x^{[\beta,1]}=\Lambda^+$ are
above the line $L$, the point $x^{[\alpha,\beta]}$ has to be below
this line and therefore in $\Omega_\ve$. This means just what we
need that the oscillation over $[\alpha,\beta]$ does not exceed $\ve$.

It remains to note that the measure of the set where $\vf\ge\lambda$
is $1-a$, i.e. in $\Omega_2$ we have
$$
\Bel\ge 1-a=
1-\frac{x_2+\lambda^2-2\lambda x_1-2\ve(x_1-\lambda)}{8\ve^2}\,.
$$

To get an extremizer for a point $x$ on the intersection of the
upper parabola with $\Omega_4^+$ we need to concatenate the
logarithmic function with the step function corresponding to the
upper right corner of $\Omega_4^+$, i.e. with the step function
consisting of two steps of equal size with the values $\lambda$
and $\lambda-2\ve$. For an arbitrary point $x\in\Omega_4^+$
we have cut the latter function from below on the corresponding
level. As a result we get the following
$$
\vf(t)=
\begin{cases}
\qquad\lambda,&\text{if}\ 0<t<a,
\\
\quad\lambda-2\ve,&\text{if}\ a<t<2a,
\\
\lambda-2\ve+\ve\log\frac{2a}t,&\text{if}\ 2a<t<b,
\\
\lambda-2\ve+\ve\log\frac{2a}b,&\text{if}\ b<t<1.
\end{cases}
$$
As in the previous case, we could write down two equations
$\av\vf{[0,1]}=x_1$ and $\av{\vf^2}{[0,1]}=x_2$ and solving them
to find the appropriate value of the parameters $a$ and $b$.
However it is easier to find $a$ and $b$ using other arguments
and after that simply to check that the averages have the
desired values. For this aim we consider splitting of the interval
$[0,1]$ at the point $b$. In result we get two Bellman points
$V=x^{[0,b]}$ and $U=x^{[b,1]}$. The point $U=(u,u^2)$ is on the
lower boundary, it corresponds to the constant function
$u=\lambda-2\ve+\ve\log\frac{2a}b$. The point $V$ has to be on
the on the upper boundary and the segment $[U,V]$ has to be a
segment of the extremal line passing through $x$, i.e. a segment
of the tangent line to the upper parabola. 
(We mean here the extremal lines of the solution of the corresponding 
Monge--Amp\`ere equation, which is lurking behind all our considerations.)
But it is easy to
calculate the coordinates of the points of intersection the
tangent line to the upper parabola passing through the point $x$:
$$
u=x_1-\ve+\sqrt{\ve^2-x^2+x_1^2}\,,
$$
whence
$$
\log\frac{2a}b=1+\frac{x_1-\lambda}\ve+\sqrt{1-\frac{x_2-x_1^2}{\ve^2}}\,.
$$
Furthermore, the length of the horizontal projection of $[U,V]$ is
just $\ve$, i.e. the splitting ratio is
$$
b=\frac{x_1-u}\ve=1-\sqrt{1-\frac{x_2-x_1^2}{\ve^2}}\,,
$$
and finally
$$
a=\frac e2\left(1-\sqrt{1-\frac{x_2-x_1^2}{\ve^2}}\right)
\exp\left\{\frac{x_1-\lambda}\ve+\sqrt{1-\frac{x_2-x_1^2}{\ve^2}}\right\}.
$$
We omit verification that for this parameters $a$ and $b$ averages
of $\vf$ and $\vf^2$ have the prescribed values. To finish our proof
of the desired estimate
$$
\Bel(x)\ge a
$$
for any $x\in\Omega_4$, it remains to verify that the norm of our
test function $\vf$ does not exceed $\ve$. Again this verification
will be geometric. Consider the following curve in $\Omega_\ve$ built 
by using $\vf$ mentioned above:
$$
\psi(t)=x^{[0,t]},\qquad t\in[0,1].
$$
For $t\in[0,a]$ the point $\psi(t)$ stands at
$\Lambda=(\lambda,\lambda^2)$. At the moment $t=a$ it starts to
move to the left along the tangent line to the upper boundary.
At the moment $t=2a$ it reaches the upper parabola and continue
its movement along this upper boundary till the point $V$. It reaches
$V$ at the moment $t=b$ and then continues along $[U,V]$. The
destination point is $\psi(1)=x$. Note that this curve is convex.
Take now an arbitrary subinterval $[\alpha,\beta]\subset[0,1]$ and
draw a straight line $L$ passing through $\psi(\beta)$ and tangent to
our curve $\psi$ (i.e. tangent to the upper parabola). Since $\psi$ is
concave, the point $\psi(\alpha)=x^{[0,\alpha]}$ is above $L$ (more 
precisely, not below $L$). And we conclude that the point $x^{[\alpha,\beta]}$ 
has to be below $L$ (more precisely, not
above $L$), because the point $\psi(\beta)$ (on $L$)
is a convex combination of the point $\psi(\alpha)$ (above $L$) and
$x^{[\alpha,\beta]}$. Therefore, the latter point is in $\Omega_\ve$,
i.e. the oscillation of $\vf$ over this interval does not exceed
$\ve$.

Finally, we have to consider the most difficult case $x\in\Omega_5$.
We shall proceed as in the triangle domain $\Omega_2^+$. Arbitrary
point of $\Omega_5$ is a convex combination of three points: the
origin and $E^\pm=(\pm\ve,2\ve^2)$. Since $E^\pm\in\Omega_4^\pm$, we
already know the extremizers for these points, but for the origin
there is the only test function, namely, the constant zero function.
We concatenate these three function in the proper order (to get a
monotonous function in result). This will be the desired extremizer:
$$
\vf(t)=
\begin{cases}
\qquad-\lambda,&\text{if}\ 0<t<a_-,
\\
\quad-\lambda+2\ve,&\text{if}\ a_-<t<2a_-,
\\
\ve\log\frac{t}{\,2a_-}-\lambda+2\ve,&\text{if}\ 2a_-<t<b_-,
\\
\qquad\quad 0,&\text{if}\ b_-<t<1-b_+,
\\
\ve\log\frac{\,2a_+}{1-t}+\lambda-2\ve,&\text{if}\ 1-b_+<t<1-2a_+,
\\
\quad\phantom{-}\lambda-2\ve,&\text{if}\ 1-2a_+<t<1-a_+,
\\
\qquad\phantom{-}\lambda,&\text{if}\ 1-a_+<t<1.
\end{cases}
$$
The continuity of $\vf$ at the points $t=b_-$ and $t=1-b_+$ yields
$$
\frac{b_-}{2a_-}=\frac{b_+}{2a_+}=\exp\big(\frac\lambda\ve-2\big).
$$
From the representation
$$
x=b_-E^-+b_+E^++(1-b_--b_+){\mathbf0}
$$
we get two equations for $b_\pm$:
\begin{align*}
x_1&=-\ve b_-+\ve b_+,
\\
x_2&=2\ve^2 b_-+2\ve^2 b_+,
\end{align*}
whence
$$
b_\pm=\frac{x_2\pm2\ve x_1}{4\ve^2}\,,
$$
and therefore,
$$
a_\pm=\frac12b_\pm\exp\big(2-\frac\lambda\ve\big)
=\frac{x_2\pm2\ve x_1}{8\ve^2}\exp\big(2-\frac\lambda\ve\big)\,.
$$
Again we omit verification that $\av\vf{[0,1]}=x_1$ and 
$\av{\vf^2}{[0,1]}=x_2$, we only say few words how to check
that the norm of $\vf$ does not exceed $\ve$. We shall proceed
as in the triangle domain $\Omega_2^+$. Take any straight line
$L$ passing through $x$ and not intersecting the upper parabola.
Note that we need to consider the oscillation of $\vf$ only over
intervals $[\alpha,\beta]$ containing $[b_-,1-b_+]$, because in
other case $\vf$ on $[\alpha,\beta]$ is a part of test function
considered for the domain $\Omega_4$. Our point $x$
is a convex combination of three Bellman points $x^{[0,\alpha]}$,
$x^{[\alpha,\beta]}$, and $x^{[\beta,1]}$. It is clear that the
points $x^{[0,\alpha]}$ and $x^{[\beta,1]}$ are above the line $L$
(they are somewhere on the left and right curves considered for
the points from $\Omega^4_\pm$). Therefore, the point
$x^{[\alpha,\beta]}$ has to be below the line $L$, i.e. in
$\Omega_\ve$. This means just what we need that the oscillation
over $[\alpha,\beta]$ does not exceed $\ve$.

It remains to note that the measure of the set where $\vf\ge\lambda$
is $a_+$ and the measure of the set where $\vf\le-\lambda$
is $a_-$, i.e. in $\Omega_5$ we have
$$
\Bel\ge a_-+a_+=
\frac{x_2}{4\ve^2}\exp\big(2-\frac\lambda\ve\big)\,.
$$

This completes the proof of formula~\eqref{JNBel3}. Extremizers
for all other cases of Theorem~\ref{WJNBel} and Theorem~\ref{1jump}
are absolutely similar to those just built.
\qed

\section{How to find the expression for the Bellman function\\
and formulas for extremizers}

The theorems presented in this paper were proved in 2006, when
the problem of finding a Bellman function was a kind of art.
Using some heuristic arguments the whole domain was splitting
in several subdomains, thereafter the corresponding boundary
value problem for the homogeneous Monge--Amp\`ere equation
was solved. The solutions were glued together continuously to
get a locally convex function in the entire domain. After that,
using known foliation of the domain by the extremal lines of the solution of the Monge--Amp\`re equation, the
extremizers were constructed for every point of the domain.
{\color{yellow}
The pieces of such an approach can be found in  in \cite{SV}, \cite{VV}, \cite{VoPr}, \cite{VV1}.
The latter paper has a lengthy explanation of extremal lines  of the solutions of the Monge--Amp\`ere equation, and their pertinence to the best constant problems of Harmonic Analysis. }
But nowadays this is already an elaborated machinery. For
sufficiently smooth boundary values all of these is already written
(see~\cite{IOSVZ}). From there it is absolutely clear how to proceed
in more general situation and the corresponding text will appear
soon. By this reason we omit here any explanation about method
of finding these Bellman function --- the description of the original
way of reasoning has no sense, but to describe here the modern state
of the theory is impossible, because it would require enormous amount
of place. We refer the reader to two papers~\cite{IOSVZ}
and~\cite{SV1} for explanation of methods of solving Monge--Amp\`ere 
equation in the parabolic strip, and  to~\cite{VV} for more general cases.

The same can be said about of finding extremal test functions and
especially about the proof that the found function has the desired
$\BMO$-norm. The geometric method of proving that the $\BMO$-norm of
the extremizers does not exceed $\ve$ first appeared
in~\cite{SV1} for some special cases and then was generalized
in~\cite{IOSVZ}, where the notion of {\it delivery curves} appeared.
Traces of this notion the reader can see in the presented proof.
We have to say that this part of the proof is modern, not the
original one. The calculation of the $\BMO$-norms of extremizers
in 2006 was made by the straightforward calculation. These were
awful calculations, enormous amount of calculations. There were
impossible to place them in any paper. Maybe, that was one of the
reasons why this result was prepared for publication five years
after it was proved.


\begin{thebibliography}{XXX}
\bibitem[Bu1]{Bu1}
{\sc D.~Burkholder}, {\em Boundary value problems and sharp estimates for the
martingale transforms}, Ann. of Prob. {\bf 12} (1984), 647--702.
\bibitem{IOSVZ}
P. Ivanishvili, N. Osipov, D. Stolyarov, V. Vasyunin, P. Zatitskiy.
{\em Bellman functions for the extremal problems on $\BMO$.}
(in Russian) Preprint PDMI no.~19, 2011, 1--102.
(http://www.pdmi.ras.ru/preprint/2011/rus-2011.html)
\bibitem{JN}
F. John, L. Nirenberg. {\em On functions of bounded mean oscillation.}
Comm. Pure Appl. Math., Vol.~14, 1961, pp.~415--426.
\bibitem{K}
A. Korenovskii. {\em The connection between mean oscillations and
exact exponents of summability of functions.} (Russian) Mat. Sb.,
Vol.~181, 1990, no.~12, pp.~1721--1727; English transl. in Math.
USSR-Sb., Vol.~71, 1992, no.~2, pp.~561--567.
\bibitem{NT}
F. Nazarov, S. Treil. {\em The hunt for Bellman function:
applications to estimates of singular integral operators and
to other classical problems in harmonic analysis}, (Russian)
Algebra i Analiz, Vol.~8, 1996, no.~5, pp.~32--162; English transl.
in St.~Petersburg Math. J., Vol.~8, 1997, no.~5, 721--824.
\bibitem{NTV}
F. Nazarov, S. Treil, A. Volberg. {\em Bellman function in Stochastic
Optimal Control and Harmonic Analysis \textup(how our Bellman function
got its name\textup)}, Oper. Theory: Advances and Appl. Vol.~129, 2001, pp.~393--424.
\bibitem{NTV1}
{\sc F.~Nazarov, S.~Treil, A.~Volberg}, {\em The Bellman functions and
two-weight inequalities for Haar multipliers}, J. of Amer. Math. Soc., 12
(1999), 909-928.
\bibitem{SSV} L. Slavin, A. Stokolos, V. Vasyunin.
{\em Monge--Amp\`ere equations and Bellman functions: the dyadic
maximal operator}, Comptes Rend. Math., Vol.~346, Ser.~I, 2008,
pp.~585--588.
\bibitem{SV}{L. Slavin, V. Vasyunin},
{\em Sharp results in the integral-form John--Nirenberg inequality},
Trans. Amer. Math. Soc., Vol.~363 (2011), No.~8, pp.~4135--4169.
(Preprint, 2007; http://arxiv.org/abs/0709.4332)
\bibitem{SV1}{L. Slavin, V. Vasyunin},
{\em Sharp $L^p$-estimates on $\BMO$},
Indiana University Math. J.(to appear);
http://www.iumj.indiana.edu/IMJU/Preprints/4651.pdf.
\bibitem{T}
T.~Tao. {\em Bellman function and the John--Nirenberg inequality.} Preprint, (http://www.math.ucla.edu/\~{}tao/harmonic.html).
\bibitem{V}
V. Vasyunin. The sharp constant in the John--Nirenberg inequality, Preprint PDMI no.~20, 2003; http://www.pdmi.ras.ru/preprint/2003/index.html.
\bibitem{VV}
V. Vasyunin, A. Volberg. {\em Monge--Amp\`ere equation and Bellman
optimization of Carleson Embedding Theorems},
Translations of the American Mathematical Society, Vol.~226,
pp.~ 195--238, 2009.
(Preprint, 2008; http://arxiv.org/abs/0803.2247)
\bibitem{VV1}
V. Vasyunin, A. Volberg. {\em
Burkholder's function via Monge--Ampre equation}
To appear in D. Burkholder's anniversary issue of Ill. J. of Math.
(Preprint arXiv:1006.2633)
\bibitem
{VoEc} {\sc A. Volberg}, {\em Bellman approach
to some problems in Harmonic Analysis}, S\'eminaires des Equations
aux deriv\'ees partielles. Ecole Polit\'echnique, 2002, expos\'e
XX, pp. 1--14.
\bibitem{VoPr}
{\em Bellman function technique in Harmonic Analysis. Lectures of INRIA Summer School in Antibes}
(Preprint, pp. 1--58, 2011;  arXiv:1106.3899 )
\end{thebibliography}
\end{document}